\documentclass[12pt,reqno,twoside]{amsart}
\usepackage{graphicx}
\usepackage{amssymb}
\usepackage{epstopdf}
\usepackage[asymmetric,top=3.5cm,bottom=4.3cm,left=3.1cm,right=3.1cm]{geometry}
\usepackage{hyperref}

\usepackage{booktabs} 
\usepackage{array} 
\usepackage{paralist} 
\usepackage{verbatim} 
\usepackage{subfig} 
\usepackage{tabularx}
\usepackage{amsmath,amsfonts,amsthm,mathrsfs,amssymb,cite}
\usepackage[usenames]{color}
\usepackage{bm}
\usepackage{color}

\newtheorem{thm}{Theorem}[section]

\newtheorem{lem}{Lemma}[section]

\theoremstyle{definition}

\theoremstyle{remark}

\newtheorem{rem}{Remark}[section]
\numberwithin{equation}{section}

\newcommand{\bu}{\mathbf{u}}
\newcommand{\bv}{\mathbf{v}}

\newcommand{\bff}{\mathbf{f}}

\newcommand{\bvarphi}{\bm{\varphi}}
\newcommand{\bnu}{\bm{\nu}}

\newcommand{\norm}[1]{\left\Vert#1\right\Vert}

\newcommand{\GG}{\Gamma}

\newcommand{\Bx}{\mathbf{x}}
\newcommand{\By}{\mathbf{y}}

\newcommand{\bdelta}{\bm{\delta}}
\newcommand{\bE}{\mathbf{E}}

\newcommand{\Kcal}{\mathcal{K}}

\newcommand{\Scal}{\mathcal{S}}

\newcommand{\Mcal}{\mathcal{M}}

\newcommand{\RR}{\mathbb{R}}

\newcommand{\eqnref}[1]{(\ref {#1})}

\newcommand{\p}{\partial}

\newcommand{\beq}{\begin{equation}}
\newcommand{\eeq}{\end{equation}}

\DeclareMathAlphabet{\itbf}{OML}{cmm}{b}{it}

\def\bu{{{\itbf u}}}
\def\bv{{{\itbf v}}}

\title[Spectrum of N-P operator in elastostatics and plasmonic resonances]{On spectral properties of Neuman-Poincar\'e operator and plasmonic resonances in 3D elastostatics}

\author{Youjun Deng}
\address{School of Mathematics and Statistics, Central South University, Changsha, Hunan, P. R. China.}
\email{youjundeng@csu.edu.cn, dengyijun\_001@163.com}

\author{Hongjie Li}
\address{Department of Mathematics, Hong Kong Baptist University, Kowloon Tong, Hong Kong SAR.}
\email{hongjie$_-$li@yeah.net}

\author{Hongyu Liu}
\address{Department of Mathematics, Hong Kong Baptist University, Kowloon Tong, Hong Kong SAR.\vspace*{-4mm}}
\address{\vspace*{-4mm}and}
\address{HKBU Institute of Research and Continuing Education, Virtual University Park, Shenzhen, P. R. China.}
\email{hongyu.liuip@gmail.com}

\begin{document}
\maketitle

\begin{abstract}

We consider plasmon resonances and cloaking for the elastostatic system in $\mathbb{R}^3$ via the spectral theory of Neumann-Poincar\'e operator. We first derive the full spectral properties of the Neumann-Poincar\'e operator for the 3D elastostatic system in the spherical geometry. The spectral result is of significant interest for its own sake, and serves as a highly nontrivial extension of the corresponding 2D study in \cite{AJKKY15}. The derivation of the spectral result in 3D involves much more complicated and subtle calculations and arguments than that for the 2D case. Then we consider a 3D plasmonic structure in elastostatics which takes a general core-shell-matrix form with the metamaterial located in the shell. Using the obtained spectral result, we provide an accurate characterisation of the anomalous localised resonance and cloaking associated to such a plasmonic structure.

\medskip

\medskip

\noindent{\bf Keywords:}~~anomalous localized resonance, plasmonic material, negative elastic materials, elastostatics

\noindent{\bf 2010 Mathematics Subject Classification:}~~35B34; 74E99; 74J20

\end{abstract}

\section{Introduction}

\subsection{Background}

There has been growing interest in the mathematical study of plasmon materials. Those are a particular class of metamaterials that allow the presence of negative material parameters such as negative permittivity and permeability in electromagnetism, and negative density and refractive index in acoustics, etc.. Plasmon materials, a.k.a negative materials can find many important applications in science and technology including imaging resolution enhancement, invisibility cloaking and energy harvesting. We refer to \cite{Acm13,Ack13,AK,Bos10,Brl07,CKKL,Klsap,LLL,GWM1,GWM2,GWM3,GWM4,GWM6,GWM7,GWM8} for the relevant study in electrostatics governed by the Laplace equatioin, \cite{ADM,AMRZ,AKL,KLO} in acoustics governed by the Helmholtz equation and \cite{ARYZ} in electromagnetism governed by the Maxwell system.

Mathematically, the negativity of the material parameters breaks the ellipticity of the underlying PDO (partial differential operator). The non-elliptic PDO may possess non-trivial kernel which in turn induces resonance. The resonance is usually associated to an infinite-dimensional kernel of the non-elliptic PDO, and hence is referred to as {\it anomalous}. It is also interesting that the resonant field demonstrates a highly oscillatory behaviour, evidenced by the blowup of the associated energy of the underlying system. Furthermore, such a blowup behaviour is localised within a specific region with a sharp boundary not defined by any discontinuities in the material parameters, and the field outside that region converges to a smooth one as the loss parameter goes to zero. Due to those distinct features, it is referred to as the {\it anomalous localised resonance} in the literature. Another surprisingly interesting feature of the plasmonic resonance is that it strongly depends on the location of the forcing source.

Recently, the plasmon resonances were investigated for the linear elasticity governed by the Lam\'e system, where negative shear and bulk modulus are allowed to present \cite{AJKKY15,AKKY16,KM,LLBW,LiLiu2d,LiLiu3d}. In \cite{LiLiu2d,LiLiu3d}, the plasmon resonances in both two and three dimensions were considered. The argument is based on a variational argument using the primal and dual variational principles for the Lam\'e system. However, only energy blowup and dependence on the source location were shown by using the variational approach, and the localising and cloaking effects cannot be shown. In \cite{AJKKY15,AKKY16}, much more accurate characterisation of the anomalous localised resonance in the linear elasticity was established based on a spectral approach using the spectral properties of the Neumann-Poincar\'e (N-P) operator associated with the Lam\'e system. However, the corresponding study was only conducted in two dimensions and the major obstacle for the extension to the three-dimensional case is the lack of the spectral properties of the N-P operator associated with the Lam\'e system in $\mathbb{R}^3$. Here, we would like to emphasise that the N-P operator for the linear elasticity is not compact even on smooth domains. The derivation of the spectral properties of the two-dimensional N-P operator in \cite{AJKKY15} is highly technical. One of the main aims of the present paper is to derive the full spectral properties of the three-dimensional N-P operator for the linear elasticity in spherical geometry. As shall be seen that the derivation involves much more complicated and subtle calculations and arguments than that for the two-dimensional case. Then we consider a 3D plasmonic structure in elastostatics which takes a general core-shell-matrix form with the metamaterial located in the shell. Using the obtained spectral result, we provide an accurate characterisation of the anomalous localised resonance and cloaking associated to such a plasmonic structure.

\subsection{Mathematical setup}

For self-containedness, we next briefly introduce the mathematical formulation of the elastostatic system and the cloaking due to anomalous localised resonance in elastostatics. We also refer to \cite{AJKKY15,AKKY16,KM,LLBW,LiLiu2d,LiLiu3d} for more relevant discussions.

 Let $\mathbf{C}(\Bx):=(\mathrm{C}_{ijkl}(\Bx))_{i,j,k,l=1}^3$, $\Bx\in\mathbb{R}^3$, be a four-rank tensor such that
 \begin{equation}\label{eq:lame_constant}
 \mathrm{C}_{ijkl}(\Bx):=\lambda(\Bx)\bm{\delta}_{ij}\bdelta_{kl}+\mu(\Bx)(\bdelta_{ik}\bdelta_{jl}+\bdelta_{il}\bdelta_{jk}),\ \ \Bx\in\mathbb{R}^3,
 \end{equation}
 where $\lambda, \mu\in\mathbb{C}$ are complex-valued functions, and $\bdelta$ is the Kronecker delta. $\mathbf{C}(x)$ describes an isotropic elastic material tensor distributed in the space $\mathbb{R}^3$, where $\lambda$ and $\mu$ are referred to as the Lam\'e constants. For a regular elastic material, the Lam\'e constants are real-valued and satisfy the following strong convexity condition,
 \begin{equation}\label{eq:convex}
 \mu>0\qquad\mbox{and}\qquad 3\lambda+2\mu>0.
 \end{equation}
 For a plasmonic elastic material, the Lam\'e constants are allowed to be complex-valued with the real parts being negative and the imaginary parts signifying the loss parameters. In the sequel, we write $\mathbf{C}_{\mathbb{R}^3,\lambda,\mu}$ to specify the dependence of the elastic tensor on the Lam\'e parameters $\lambda$ and $\mu$, and the domain of interest $\mathbb{R}^3$. We shall also simply write $\mathbf{C}_{\lambda,\mu}$ for $\mathbf{C}_{\mathbb{R}^3,\lambda,\mu}$ if no confusion would arise in the context.

 Let $\Sigma$ and $\Omega$ be bounded domains in $\mathbb{R}^N$ with connected Lipschitz boundaries such that $\Sigma\Subset\Omega$. Let the elastic tensor $\mathbf{C}_{\lambda,\mu}$ be regular in the core $\Sigma$ and in the matrix $\mathbb{R}^3\backslash\overline{\Omega}$, whereas in the shell $\Omega\backslash\overline{\Sigma}$, the elastic tensor is plasmonic. We also assume that in the shell,
 \begin{equation}\label{eq:assum1}
 \Im\lambda(\Bx)=\Im\mu(\Bx)=\delta\quad\mbox{for}\ \ \Bx\in\Omega\backslash\overline{\Sigma},
 \end{equation}
 where $\delta\in\mathbb{R}_+$ is sufficiently small, signifying the lossy parameter.

Let $\mathbf{f}$ be an $\mathbb{R}^3$-valued function that is compactly supported outside $\Omega$ with a zero average,
\begin{equation}\label{eq:source1}
\int_{\mathbb{R}^3}\mathbf{f}(\Bx)\ d\Bx=0.
\end{equation}
$\mathbf{f}$ signifies an elastic source/forcing term.

Let $\mathbf{\bu}_\delta(x)\in\mathbb{C}^3$, $\Bx\in\mathbb{R}^3$, denote the displacement field in the space that is occupied by the elastic configuration $(\mathbf{C}_{{\lambda},{\mu}},\mathbf{f})$ described above. In the quasi-static regime, $\mathbf{\bu}_\delta(x)\in H_{\mathrm{loc}}^1(\mathbb{R}^3)^3$ verifies the following Lam\'e system
\begin{equation}\label{eq:lame1n1}
\begin{cases}
&\mathcal{L}_{\lambda,\mu} \mathbf{\bu}_\delta(\Bx)=\mathbf{f}(\Bx),\quad \Bx\in\mathbb{R}^3,\medskip\\
& \mathbf{\bu}_\delta|_-=\mathbf{\bu}_\delta|_+,\quad \partial_{\bm{\nu}_{\lambda,\mu}} \mathbf{\bu}_{\delta}|_-=\partial_{\bm{\nu}_{\lambda,\mu}} \mathbf{\bu}_{\delta}|_+\quad\mbox{on}\ \ \partial\Sigma\cup\partial\Omega,\medskip\\
&\mathbf{\bu}_\delta(\Bx)=\mathcal{O}\big(|\Bx|^{-1}\big)\quad\mbox{as}\ \ |\Bx|\rightarrow+\infty,
\end{cases}
\end{equation}
where the PDO $\mathcal{L}_{{\lambda},{\mu}}$ is defined by
\begin{equation}\label{eq:lame2}
\mathcal{L}_{{\lambda},{\mu}} \mathbf{\bu}_\delta:=\nabla\cdot\mathbf{C}_{{\lambda},{\mu}} {\hat{\nabla}}\mathbf{\bu}_\delta={\mu}\Delta\mathbf{\bu}_\delta+({\lambda}+{\mu})\nabla\nabla\cdot\mathbf{\bu}_\delta,
\end{equation}
with $\hat{\nabla}$ signifying symmetric gradient
\[
\hat{\nabla}\mathbf{\bu}_\delta:=\frac 1 2(\nabla\mathbf{\bu}_\delta+\nabla\mathbf{\bu}_\delta^t),
\]
and the superscript $t$ denoting the matrix transpose. In \eqref{eq:lame1n1}, the conormal derivative (or traction) is defined by
\begin{equation}\label{eq:normald}
\partial_{\bm{\nu}}\mathbf{\bu}_\delta=\frac{\partial \bu_\delta}{\partial \bm{\nu}}:= \lambda(\nabla \cdot \ \bu_\delta)\bm{\nu}+ \mu(\nabla \bu_\delta + \nabla \bu_\delta^t)\bm{\nu}\quad \mbox{on}\ \ \partial\Sigma\ \mbox{or}\ \partial \Omega,
\end{equation}
where $\bm{\nu}$ denotes the exterior unit normal to $\partial \Sigma/\partial \Omega$,
and the $\pm$ signify the traces taken from outside and inside of the domain $\Sigma/\Omega$, respectively.

Next, for $\mathbf{\bu}\in H^1_{\text{loc}}(\mathbb{R}^3)^3$ and $\mathbf{\bv}\in H^1_{\text{loc}}(\mathbb{R}^3)^3$, we introduce
\begin{equation}\label{eq:energy1}
\mathbf{P}_{\lambda,\mu}(\mathbf{\bu},\mathbf{\bv}):=\int_{\mathbb{R}^3}\big[\lambda(\nabla\cdot\mathbf{\bu})\overline{(\nabla\cdot\mathbf{\bv})}(\Bx)+2\mu\nabla^s\mathbf{u}:\overline{\nabla^s\mathbf{\bv}}(\Bx) \big]\ d \Bx,
\end{equation}
where and also in what follows, $\mathbf{A}:\mathbf{B}=\sum_{i,j=1}^3 a_{ij}b_{ij}$ for two matrices $\mathbf{A}=(a_{ij})_{i,j=1}^3$ and $\mathbf{B}=(b_{ij})_{i,j=1}^3$.
For the solution $\mathbf{u}_\delta$ to \eqref{eq:lame1n1}, we define
\begin{equation}\label{eq:energy2}
\mathbf{E}_\delta(\mathbf{C}_{{\lambda},{\mu}}, \mathbf{f}):=\frac{\delta}{2}\mathbf{P}_{\lambda,\mu}(\mathbf{\bu}_\delta, \mathbf{\bu}_\delta).
\end{equation}
Then cloaking due to anomalous localised resonance occurs if the following two conditions are satisfied:
\begin{equation}\label{eq:condition_CALR}
  \left\{
    \begin{array}{ll}
     \displaystyle{\limsup_{\delta\rightarrow+0}}\  \mathbf{E}(\bu_{\delta})\rightarrow \infty,\medskip \\
      |\bu_{\delta}(\Bx)|<C , \quad  |\Bx|>r',
    \end{array}
  \right.
\end{equation}
for some constants $C$ and $r'$ independent of $\delta$.

The rest of the paper is organised as follows. In Section 2, we present some preliminary knowledge on layer potential operators including the Neumann-Poincar\'e (N-P) operator for the Lam\'e system \eqref{eq:lame1n1}. Section 3 is devoted to the spectral properties of the N-P operator in spherical geometry. In Section 4, we consider the cloaking due to anomalous localised resonance in elastostatics. 

\section{Preliminaries on layer potentials}

We first introduce some function spaces that shall be needed in our subsequent study. Let $D$ be a bounded Lipscthiz domain with a connected complement in $\RR^3$. Let  $\nabla_{\p D}\cdot$ denote the surface divergence. Denote by $L_T^2(\p D):=\{\bvarphi\in {L^2(\p D)}^3, \bnu\cdot \bvarphi=0\}$. Let $H^s(\partial D)$ be the usual Sobolev space of order $s\in\mathbb{R}$ on $\partial D$. We also introduce the function spaces
\begin{align*}
\mathrm{TH}({\rm div}, \p D):&=\Bigr\{ {\bvarphi} \in L_T^2(\partial D):
\nabla_{\partial D}\cdot {\bvarphi} \in L^2(\partial D) \Bigr\},\\
\mathrm{TH}({\rm curl}, \p D):&=\Bigr\{ {\bvarphi} \in L_T^2(\partial D):
\nabla_{\partial D}\cdot ({\bvarphi}\times {\bnu}) \in L^2(\partial D) \Bigr\},
\end{align*}
equipped with the norms
\begin{align*}
&\|{\bvarphi}\|_{\mathrm{TH}({\rm div}, \p D)}=\|{\bvarphi}\|_{L^2(\p D)}+\|\nabla_{\p D}\cdot {\bvarphi}\|_{L^2(\p D)},\medskip \\
&\|{\bvarphi}\|_{\mathrm{TH}({\rm curl}, \p D)}=\|{\bvarphi}\|_{L^2(\p D)}+\|\nabla_{\p D}\cdot({\bvarphi}\times \bnu)\|_{L^2(\p D)}.
\end{align*}

In the following, we introduce some basic notions on layer potentials. For a density function $\phi$, denote by $\Scal_D: H^{-1/2}(\p D)^m\rightarrow H^{1/2}(\p D)^m$, $m=1,3$ the single layer potential operator, which is represented as follows
\beq\label{eq:layerpt1}
\Scal_{D}[\phi](\Bx):=\int_{\p D}\Gamma(\Bx-\By)\phi(\By)d s_\By,
\eeq
where $\GG(\Bx)$ is the fundamental solution to Laplacian $\Delta$ given by
\beq\label{eq:funda01}
\GG(\Bx)=-\frac{1}{4\pi|\Bx|}.
\eeq
It is known that the single layer potential operator $\Scal_D$ satisfies the trace formula
\beq \label{eq:trace}
\frac{\p}{\p\bnu}\Scal_D[\phi] \Big|_{\pm} = \pm \frac{\phi}{2}+
\Kcal_{D}^*[\phi] \quad \mbox{on } \p D, \eeq
where $(\Kcal_{D})^*$ is the adjoint operator of $\Kcal_D$.
We mention that the density function $\phi$ can be either a scalar density function or a vector density function in $\RR^3$. If $\phi\in L_T^2 (\p D)$  then $\Scal_D[\phi]$ is continuous on $\RR^3$ and its curl satisfies the following jump formula:
\begin{equation}\label{jumpM}
\bnu \times \nabla \times \Scal_D[\phi]\big\vert_\pm = \mp \frac{\phi}{2} + \Mcal_D[\phi] \quad \mbox{ on } \p D,
\end{equation}
where $\Mcal_D$ is the boundary operator defined by
\begin{equation}\label{defM}
\begin{aligned} \mathcal{M}_D:
\mathrm{L}_T^2 (\partial D)  & \longrightarrow \mathrm{L}_T^2 (\partial D)  \\
\phi & \longmapsto \mathcal{M}_D[\phi](\Bx)= \bnu_\Bx  \times \nabla_\Bx \times \int_{\p D} \Gamma(\Bx,\By) \bnu_\By \times \phi(\By) ds(\By).
\end{aligned}
\end{equation}

On the other hand, for a vector function $\bvarphi$ on $\p D$, denote by $\mathbf{S}_{D}[\bvarphi](\Bx)$ the single layer potential associated with the Lam\'e system \eqref{eq:lame1n1},
\begin{equation}\label{eq:layerpt2}
  \mathbf{S}_{D}[\bvarphi](\Bx) := \int_{\partial D} \mathbf{G}(\Bx-\By)\bvarphi(\By)ds(\By), \quad \Bx\in \mathbb{R}^3 \backslash \partial{D},
\end{equation}
where $\mathbf{G} = (G_{j,k})^3_{j,k=1}$ is the Kelvin matrix of fundamental solutions to the Lam\'e operator $\mathcal{L}_{\lambda, \mu}$ and has the following representation
\begin{equation}\label{eq:fundamentalsolution_0}
  G_{j,k}(\Bx)=-\frac{\alpha_1}{4\pi} \frac{\delta_{jk}}{|\Bx|} -\frac{\alpha_2}{4\pi} \frac{\Bx_j \Bx_k}{|\Bx|^3},
\end{equation}
with
\begin{equation}
  \alpha_1:=\frac{1}{2}\left( \frac{1}{\mu} + \frac{1}{2\mu+\lambda} \right) \quad \mbox{and} \quad \alpha_2:=\frac{1}{2}\left( \frac{1}{\mu} - \frac{1}{2\mu+\lambda}\right).
\end{equation}
From the definition of traction in \eqref{eq:normald}, the vector valued single layer potential \eqnref{eq:layerpt2} enjoys the following jump relation
\begin{equation}\label{eq:jump_single}
  \frac{\partial}{\partial \bnu} \mathbf{S}_D[\bvarphi]|_{\pm}(\Bx)=\left( \pm \frac{1}{2}\mathbf{I}_3 + \mathbf{K}^*_D \right) [\bvarphi](\Bx), \quad \mbox{a.e.} \; \Bx\in \partial D,
\end{equation}
where $\mathbf{I}_3$ denotes the identity matrix operator in $\RR^3$ and $\mathbf{K}^*_D$ is the Neumann-Poincar\'e(N-P) operator defined by
\begin{equation}\label{eq:operator_k_star}
  \mathbf{K}^*_D[\bvarphi](\Bx):=\mbox{p.v.}\quad \int_{\partial D} \frac{\partial}{\partial\bnu_\Bx} \mathbf{G}(\Bx-\By)\bvarphi(\By)ds(\By).
\end{equation}
In \eqref{eq:operator_k_star}, p.v. stands for the Cauchy principal value. Here and also what in follows, $\frac{\partial}{\partial\bnu_\Bx} \mathbf{G}(\Bx-\By)\bvarphi(\By)$ is defined by
$$\frac{\partial}{\partial\bnu_\Bx} \mathbf{G}(\Bx-\By)\bvarphi(\By):=\frac{\partial}{\partial\bnu_\Bx} (\mathbf{G}(\Bx-\By)\bvarphi(\By)).$$

\section{Spectral analysis of N-P operator in spherical geometry}

In this section, we shall derive the spectral of the N-P operator, $\mathbf{K}^*_D$ associated with Lam\'e system on a ball.
It has been pointed out that the  $\mathbf{K}^*_{D}$ is not a compact operator even if the domain $D$ has a smooth boundary (\cite{AKKY16}), thus we cannot infer directly that the N-P operator has point spectrum on a general smooth domain. However when $D$ is a ball, the properties  of $\mathbf{K}^*_D$ is more elaborate. In this paper, we shall derive the eigenvalues of the N-P operator $\mathbf{K}^*_{D}$ and its corresponding eigenfunctions when the domain $D$ is a ball. Before this, we present several auxiliary lemmas.
\begin{lem}\label{le:01}
Suppose $D$ is a central ball in $\RR^3$ with radius $r_0$. Then the N-P operator $\mathbf{K}^*_D$ can be written in the following form
\beq\label{eq:npop01}
\begin{split}
  \mathbf{K}^*_D[\bvarphi](\Bx)=& -3\frac{\mu}{r_0}\mathbf{S}_D[\bvarphi](\Bx)+(\frac{3}{2}+\frac{\mu}{2(2\mu+\lambda)})\frac{1}{r_0}\Scal_D[\bvarphi](\Bx) \\
  &-\frac{\mu}{2\mu + \lambda}\left(\nabla\times\Scal_{D}[\bnu\times\bvarphi](\Bx)-\nabla\Scal_{D}[\bnu\cdot\bvarphi](\Bx) \right) .
  \end{split}
\eeq
\end{lem}
\begin{proof}
Let $\Bx$ and $\By$ be vectors on $\p D$. By \eqnref{eq:fundamentalsolution_0} and straightforward computations one can show that (see also \cite{AJKKY15})
 \begin{equation}\label{eq:le01}
    \partial_{\bnu_\Bx}\mathbf{G}(\Bx-\By)=-b_1 \mathbf{K}_1(\Bx,\By) + \mathbf{K}_2(\Bx,\By),
  \end{equation}
  where
  \begin{equation}\label{eq:le02}
    \begin{split}
      \mathbf{K}_1(\Bx,\By)= & \frac{\bnu_\Bx(\Bx-\By)^t - (\Bx-\By) \bnu_\Bx^t}{4\pi|\Bx-\By|^3}, \\
      \mathbf{K}_2(\Bx,\By)= & b_1 \frac{(\Bx-\By)\cdot\bnu_\Bx}{4\pi|\Bx-\By|^3}\mathbf{I}_3 + b_2 \frac{(\Bx-\By)\cdot\bnu_\Bx}{4\pi|\Bx-\By|^5}(\Bx-\By)(\Bx-\By)^t,
    \end{split}
  \end{equation}
  with
  \begin{equation}\label{eq:le03}
    b_1=\frac{\mu}{2\mu + \lambda} \quad \mbox{and}\quad b_2=\frac{3(\mu+\lambda)}{2\mu+\lambda}.
  \end{equation}
Then by \eqnref{eq:operator_k_star}, we have
\beq\label{eq:le001}
\begin{split}
\mathbf{K}^*_D[\bvarphi](\Bx)=&-b_1\int_{\p D}\mathbf{K}_1(\Bx,\By)\bvarphi(\By)ds(\By)+\int_{\p D}\mathbf{K}_2(\Bx,\By)\bvarphi(\By)ds(\By)\\
:=& L_1+L_2.
\end{split}
\eeq
Since $D$ is a central ball, for $\Bx,\By\in \p D$, one has that
$$
(\bnu_\Bx-\bnu_\By)(\Bx-\By)^t=(\Bx-\By)(\bnu_\Bx-\bnu_\By)^t
$$
and thus
\begin{equation}\label{eq:le04}
    \begin{split}
      \mathbf{K}_1(\Bx,\By)= & \frac{\bnu_\Bx(\Bx-\By)^t - (\Bx-\By) \bnu_\Bx^t}{4\pi|\Bx-\By|^3}, \\
                  = & \frac{(\bnu_\Bx-\bnu_\By+\bnu_\By)(\Bx-\By)^t - (\Bx-\By) (\bnu_\Bx -\bnu_\By+\bnu_\By)^t}{4\pi|\Bx-\By|^3}, \\
                  = & \frac{\bnu_\By(\Bx-\By)^t - (\Bx-\By) \bnu_\By^t}{4\pi|\Bx-\By|^3}.
    \end{split}
  \end{equation}
Next, it is also easy to verify that
\begin{equation}\label{eq:le05}
  \frac{(\Bx-\By)\cdot \bnu_\By}{|\Bx-\By|^3}=-\frac{1}{2r_0} \frac{1}{|\Bx-\By|}.
\end{equation}
By using vector calculus identity, \eqnref{eq:le001} and \eqnref{eq:le05}, there holds
\beq\label{eq:le002}
\begin{split}
  L_1 & =-b_1\int_{\p D}\nabla_\Bx\Gamma(\Bx-\By)\times\bnu_\By\times\bvarphi(\By)+\frac{1}{2r_0}\Gamma(\Bx-\By)\bvarphi-\nabla_\Bx\Gamma(\Bx-\By)(\bnu\cdot\bvarphi)ds(\By) \\
      & =-b_1\left(\nabla\times\Scal_{D}[\bnu\times\bvarphi](\Bx)+\frac{1}{2r_0}\Scal_{D}[\bvarphi](\Bx)-\nabla\Scal_{D}[\bnu\cdot\bvarphi](\Bx) \right)
\end{split}
\eeq
Then by direct calculation, one further has that
\begin{equation}\label{eq:K_2}
  \begin{split}
     \mathbf{K}_2(\Bx,\By) &= -\frac{b_1}{2r_0}  \Gamma(\Bx-\By)\mathbf{I}_3 + \frac{b_2}{2r_0} \frac{(\Bx-\By) (\Bx-\By)^t}{4\pi |\Bx-\By|^3}\\
      & =    -\frac{b_2}{2r_0\alpha_2} \mathbf{G}(\Bx-\By) + (\frac{b_2 \alpha_1}{2r_0\alpha_2} -\frac{b_1}{2r_0})\Gamma(\Bx-\By) \mathbf{I}_3 .
  \end{split}
\end{equation}
Hence, there holds
\beq\label{eq:le003}
\begin{split}
L_2&= -\frac{b_2}{2r_0\alpha_2}\int_{\p D} \mathbf{G}(\Bx-\By)\bvarphi(\By)ds(\By)+(\frac{b_2 \alpha_1}{2r_0\alpha_2} -\frac{b_1}{2r_0})\int_{\p D}\Gamma(\Bx-\By)\bvarphi(\By)ds(\By)\\
&=-\frac{b_2}{2r_0\alpha_2}\mathbf{S}_D[\bvarphi](\Bx)+(\frac{b_2 \alpha_1}{2r_0\alpha_2} -\frac{b_1}{2r_0})\Scal_D[\bvarphi](\Bx).
\end{split}
\eeq
Finally, by combining \eqnref{eq:le002} and \eqnref{eq:le003}, we have
\beq
\begin{split}
  \mathbf{K}^*_D[\bvarphi](\Bx)=& -b_1\left(\nabla\times\Scal_{D}[\bnu\times\bvarphi](\Bx)-\nabla\Scal_{D}[\bnu\cdot\bvarphi](\Bx) \right) \\
    & -\frac{b_2}{2r_0\alpha_2}\mathbf{S}_D[\bvarphi](\Bx)+(\frac{b_2 \alpha_1}{2r_0\alpha_2} -\frac{b_1}{r_0})\Scal_D[\bvarphi](\Bx).
\end{split}
\eeq
By calculating the coefficients in the above equation we arrive at \eqnref{eq:npop01},
which completes the proof.
\end{proof}

\begin{rem}\label{re:01}
We mention that the last two terms in \eqnref{eq:npop01} are defined by Cauchy principal values, and it is clearly that the related two operators in the last two terms are not compact operators, which shows that $\mathbf{K}_D^*$ is not a compact operator in $L^2(\p D)^3$.
\end{rem}
In the following, we shall define orthogonal vectorial polynomials which will be quite important in the analysis of the spectral of the N-P operator $\mathbf{K}_D^*$. Let $r=|\Bx|$ and $Y^m_n(\hat{\Bx})$, $-n\leq m \leq n$ be spherical harmonics on the unit sphere $S$. Define three vectorial polynomials
\begin{equation}\label{eq:elaspol02}
  \mathcal{T}^m_n(\Bx)=\nabla (r^{n} Y^m_{n}(\hat{\Bx}))\times \Bx, \quad n\geq 1, \quad  -n\leq m\leq n,
\end{equation}
and
\begin{equation}\label{eq:elaspol01}
  \mathcal{M}^m_n(\Bx)=\nabla(r^{n} Y^m_{n}(\hat{\Bx})), \quad n\geq 1,\quad  -n\leq m\leq n,
\end{equation}
and
\begin{equation}\label{eq:elaspol03}
   \mathcal{N}^m_n(\Bx)=a^m_n r^{n-1} Y^m_{n-1}(\hat{\Bx})\Bx + (1-\frac{a_n^m}{2n-1}-r^2) \nabla(r^{n-1} Y^m_{n-1}(\hat{\Bx})),
\end{equation}
where
\begin{equation}\label{eq:anmdef01}
  a^m_n = \frac{2(n-1)\lambda + 2(3n-2)\mu}{(n+2)\lambda + (n+4)\mu}, \quad n\geq 1, \quad  -(n-1)\leq m\leq n-1.
\end{equation}
By directly using the trace theorem,
the traces of $\mathcal{T}^m_n$, $\mathcal{M}^m_n$ and $\mathcal{N}^m_n$ on the unit sphere $S$, denoted by $\mathbf{T}^m_n$, $\mathbf{M}^m_n$ and $\mathbf{N}^m_n$, have the following form
\beq\label{eq:elaspol04}
  \begin{split}
    \mathbf{T}^m_n(\Bx)= & \nabla_{S}  Y^m_{n}(\hat{\Bx}) \times \bnu_\Bx,\\
        \mathbf{M}^m_n(\Bx)= & \nabla_{S} Y^m_{n}(\hat{\Bx}) +n  Y^m_{n}(\hat{\Bx}) \bnu_\Bx, \\
    \mathbf{N}^m_n(\Bx)= &\frac{a_n^m}{2n-1}(-\nabla_{S}  Y^m_{n-1}(\hat{\Bx}) + nY^m_{n-1}(\hat{\Bx})\bnu_\Bx).
  \end{split}
\eeq
We have the following fundamental result
\begin{lem}\label{le:02}
The polynomials $\mathcal{T}^m_n$, $\mathcal{M}^m_n$ and $\mathcal{N}^m_n$ are solutions to the elastic equation $\mathcal{L}_{\lambda, \mu}\bu(\Bx)=0$. Moreover, ($\mathbf{T}^m_n$, $\mathbf{M}^m_n$, $\mathbf{N}^m_n$) defined in \eqnref{eq:elaspol04} forms an orthogonal basis on $L^2(S)^3$.
\end{lem}
\begin{proof}
It is easy to find that $\mathcal{T}^m_n$ and $\mathcal{M}^m_n$ are spherical harmonic functions and divergence free (see Theorem 2.4.7 in \cite{Jcn}). Thus from \eqnref{eq:lame2} one can easily obtain that
$$
\mathcal{L}_{\lambda, \mu}\mathcal{M}^m_n=\mathcal{L}_{\lambda, \mu}\mathcal{T}^m_n=0.
$$
For $\mathcal{N}^m_n$, note that by \eqnref{eq:elaspol03} there holds
$$\nabla\cdot \mathcal{N}^m_n=(a^m_n(n+2)-2(n-1))r^{n-1}Y_{n-1}^m(\hat{\Bx}),$$
and
$$
\nabla\times\nabla\times\mathcal{N}^m_n=n(a^m_n+2)\nabla(r^{n-1}Y_{n-1}^m).
$$
Hence by using \eqnref{eq:lame2} again and \eqnref{eq:anmdef01}, one has
\beq\label{eq:Nmn01}
\mathcal{L}_{\lambda, \mu}\mathcal{N}^m_n=\mu \Delta \mathcal{N}^m_n +(\lambda + \mu)\nabla \nabla \cdot \mathcal{N}^m_n=0.
\eeq
The orthogonality of ($\mathbf{T}^m_n$, $\mathbf{M}^m_n$, $\mathbf{N}^m_n$) can be obtained by straightforward computations (cf. \cite{Jcn}).

The proof is complete.
\end{proof}

\begin{lem}\label{le:03}
Suppose that the domain $D$ is a central ball with radius $r_0$. Let ($\mathbf{T}^m_n$, $\mathbf{M}^m_n$, $\mathbf{N}^m_n$) be defined in \eqnref{eq:elaspol04}, then there holds the following on $\p D$
\beq\label{eq:le03001}
\begin{split}
\Scal_D[\mathbf{T}^m_n]&=-\frac{r_0}{2n+1}\mathbf{T}^m_n,\\
\Scal_D[\mathbf{M}^m_n]&= -\frac{r_0}{(2n-1)}\mathbf{M}^m_n,\\
\Scal_D[\mathbf{N}^m_{n+1}]&=-\frac{r_0}{2n+3}\mathbf{N}^m_{n+1},
\end{split}
\eeq
where $n\geq 0$ and $-n\leq m\leq n$.
\end{lem}
\begin{proof}
We shall only prove the second identity in \eqnref{eq:le03001} and the other two can be proved similarly. Without loss of generality we suppose that $D$ is a unit sphere. By using the jump formula \eqnref{eq:jump_single} we have
\beq\label{eq:le03002}
\frac{\p\Scal_D[\mathbf{M}^m_n]}{\p\bnu}\Big|_-=-\frac{1}{2}\mathbf{M}^m_n+\Kcal_D^*[\mathbf{M}^m_n],\quad \mbox{on} \quad\p D.
\eeq
Since $D$ is a ball, there holds the following identity (cf. \cite{HK07:book})
$$
\Kcal_D^*[\mathbf{M}^m_n]=-\frac{1}{2}\Scal_D[\mathbf{M}^m_n],
$$
and thus
\beq\label{eq:le03003}
\frac{\p\Scal_D[\mathbf{M}^m_n]}{\p\bnu}\Big|_-=-\frac{1}{2}\mathbf{M}^m_n-\frac{1}{2}\Scal_D[\mathbf{M}^m_n]\quad \mbox{on} \quad\p D.
\eeq
Suppose $\Scal_D[\mathbf{M}^m_n]$ has the following form in $D$
\beq\label{eq:le03004}
\Scal_D[\mathbf{M}^m_n]=c_1 \nabla(r^nY_n^m)+c_2\left((2n+1)r^nY_n^m\Bx-r^2\nabla(r^nY_n^m)\right) ,
\eeq
where $c_1$ and $c_2$ are constants which depends on $n$. Then by substituting \eqnref{eq:le03004} into \eqnref{eq:le03003} and using the trace theorem there holds
\beq
\begin{split}
    & c_1(n-1)\mathbf{M}^m_n + c_2(n+1)(-\nabla_{S}  Y^m_{n}(\hat{\Bx}) + (n+1)Y^m_{n}(\hat{\Bx})\bnu_\Bx) \\
   =& -\frac{1}{2}\mathbf{M}^m_n - \frac{1}{2}(c_1\mathbf{M}^m_n + c_2(-\nabla_{S}  Y^m_{n}(\hat{\Bx}) + (n+1)Y^m_{n}(\hat{\Bx})\bnu_\Bx) ),
\end{split}
\eeq
and by using the orthogonality property one has
\beq\label{eq:le03005}
\begin{split}
&c_1(n-1)=-1/2-c_1/2, \\
&c_2(n+1)=-c_2/2.
\end{split}
\eeq
By solving \eqnref{eq:le03005} we get that
\beq\label{eq:le03006}
c_1=-\frac{1}{2n-1},\quad c_2=0.
\eeq
Finally by substituting \eqnref{eq:le03006} into \eqnref{eq:le03004} and the trace theorem we obtain the first equation in \eqnref{eq:le03001}.

The proof is complete.
\end{proof}

\begin{thm}\label{thm:eigenvalue}
  Suppose that the domain $D$ is a central ball of radius $r_0$, then the the eigenvalues of the operator $\mathbf{K}^*_D$ are given by
  \begin{equation}\label{eq:eigenvalue}
    \begin{split}
      \xi^n_1= & \frac{3}{4n+2}, \\
      \xi^n_2= &\frac{3\lambda-2\mu(2n^2-2n-3)}{2(\lambda+2\mu)(4n^2-1)}, \\
      \xi^n_3= & \frac{-3 \lambda + 2\mu(2n^2 + 2n - 3)}{2(\lambda + 2\mu)(4n^2 - 1)},
    \end{split}
  \end{equation}
  where $n\geq 1$ are nature numbers, and the corresponding eigenfunctions are respectively $\mathbf{T}_n^m$, $\mathbf{M}_n^m$ and $\mathbf{N}_n^m$.
\end{thm}

\begin{proof}
Without loss of generality we set $r_0=1$. First, letting $\bvarphi=\mathbf{T}_n^m=\nabla_S Y_n^m\times\bnu$ and using the results in Lemma \ref{le:03}, one can show that
\beq\label{eq:eigenpf01}
\Scal_D[\nabla_S Y_n^m\times\bnu]=-\frac{1}{2n+1}\nabla (r^nY_n^m)\times\Bx, \quad \mbox{in} \quad D.
\eeq
Furthermore, there holds
\beq\label{eq:eigenpf02}
\nabla\times\Scal_D[\bnu\times\nabla_S Y_n^m\times\bnu]=\nabla\times\Scal_D[\nabla_S Y_n^m]=\frac{n}{2n+1}\nabla (r^nY_n^m)\times\Bx, \quad \mbox{in} \quad D.
\eeq
and by using the jump formula \eqnref{jumpM} there also holds
\beq\label{eq:eigenpf022}
\nabla\times\Scal_D[\bnu\times\nabla_S Y_n^m\times\bnu]=\frac{n}{2n+1}\nabla_SY_n^m\times\bnu-\frac{1}{2}\nabla_S Y_n^m\times\bnu, \quad \mbox{on} \quad \p D.
\eeq
Hence, by using \eqnref{eq:npop01}, \eqnref{eq:eigenpf01} and \eqnref{eq:eigenpf022} one obtains
\beq\label{eq:eigenpf03}
\mathbf{K}_D^*[\nabla_S Y_n^m\times\bnu]=-3\mu\mathbf{S}_D[\nabla_S Y_n^m\times\bnu]-\frac{3}{2(2n+1)}\nabla_S Y_n^m\times\bnu.
\eeq
By combining the jump formula \eqnref{eq:jump_single} one can suppose that
\beq\label{eq:eigenpf04}
\mathbf{S}_D[\nabla_S Y_n^m\times\bnu]=c\nabla(r^n Y_n^m)\times\Bx \quad \mbox{in}\quad D,
\eeq
and by using \eqnref{eq:normald} one can calculate
\beq\label{eq:eigenpf05}
\begin{split}
\frac{\partial}{\partial \bnu} \mathbf{S}_D[\nabla_S Y_n^m \times\bnu]|_{-}&=c\mu(\nabla(\nabla(r^n Y_n^m)\times\Bx)+(\nabla(\nabla(r^n Y_n^m)\times\Bx))^t)\bnu\\
&=c\mu(n-1)\nabla_S Y_n^m \times\bnu.
\end{split}
\eeq
By substituting \eqnref{eq:eigenpf05} into \eqnref{eq:normald} and using \eqnref{eq:eigenpf03} and \eqnref{eq:eigenpf04}, one can show
\beq\label{eq:eigenpf06}
c\mu(n-1)\nabla_S Y_n^m \times\bnu=-\frac{1}{2}\nabla_S Y_n^m \times\bnu-\left(3c\mu+\frac{3}{2(2n+1)}\right)\nabla_S Y_n^m\times\bnu.
\eeq
Therefore, we have
\beq\label{eq:eigenpf07}
c=-\frac{1}{(2n+1)\mu}.
\eeq
Finally by substituting \eqnref{eq:eigenpf07} into \eqnref{eq:eigenpf03}, one can obtain that
\beq\label{eq:eigenpf08}
\mathbf{K}_D^*[\nabla_S Y_n^m\times\bnu]=\frac{3}{2(2n+1)}\nabla_S Y_n^m\times\bnu.
\eeq

Next, by letting $\bvarphi=\mathbf{M}_n^m$, one can show that
$$
\Scal_D[\mathbf{M}_n^m]=-\frac{1}{2n-1}\nabla (r^nY_n^m) \quad \mbox{in} \quad D,
$$
and
$$
\nabla\times\Scal_D[\bnu\times\mathbf{M}_n^m]=\nabla\times\Scal_D[\bnu\times\nabla_S Y_n^m]=\frac{n+1}{2n+1}\nabla (r^nY_n^m) \quad \mbox{in} \quad D.
$$
Straightforward computations also yields that
$$
\nabla\Scal_D[\bnu\cdot\mathbf{M}_n^m]=\nabla\Scal_D[nY_n^m]=-\frac{n}{2n+1}\nabla (r^nY_n^m) \quad \mbox{in} \quad D.
$$
Then by using the jump formulas there holds
$$
\nabla\times\Scal_D[\bnu\times\mathbf{M}_n^m]-\nabla\Scal_D[\bnu\cdot\mathbf{M}_n^m]=\frac{1}{2}\mathbf{M}_n^m \quad \mbox{on} \quad \p D.
$$
Next, we assume that $\mathbf{S}_D[\mathbf{M}_n^m]=c\nabla(r^n Y_n^m)$ in $D$ and then one can show that
$$
\frac{\partial}{\partial \bnu} \mathbf{S}_D[\mathbf{M}_n^m]|_{-}=2c\mu \frac{\p}{\p\bnu}\nabla(r^nY_n^m) =2c\mu(n-1)\mathbf{M}_n^m.
$$
Hence, there holds
$$
\mathbf{K}_D^*[\mathbf{M}_n^m]=-\frac{1}{2}\mathbf{M}_n^m-3\mu\mathbf{S}_D[\mathbf{M}_n^m]-\left(\frac{3}{2(2n-1)}
+\frac{\mu n}{(2\mu+\lambda)(2n-1)}\right)\mathbf{M}_n^m.
$$
By using the jump formula, one has
$$
c\mu(2n+1)=-(\frac{1}{2}+\frac{3}{2(2n-1)}+\frac{\mu n}{(2\mu+\lambda)(2n-1)}),
$$
and thus
$$
\mathbf{K}_D^*[\mathbf{M}_n^m]=\frac{3\lambda-2\mu(2n^2-2n-3)}{2(2\mu+\lambda)(4n^2-1)}\mathbf{M}_n^m.
$$

Finally, by letting $\bvarphi=  \mathbf{N}_{n+1}^m$, one can show that
$$
\Scal_D[\mathbf{N}_{n+1}^m]=-\frac{1}{2n+3}\frac{a_{n+1}^m }{2n+1}\left((2n+1)r^nY_n^m\Bx-r^2\nabla(r^nY_n^m)\right)\quad \mbox{in} \quad D,
$$
and
$$
\nabla\times\Scal_D[\bnu\times\mathbf{N}_{n+1}^m]=-\frac{a_{n+1}^m }{2n+1}\nabla\times\Scal_D[\bnu\times\nabla_S Y_n^m]=-\frac{n+1}{2n+1}\frac{a_{n+1}^m }{2n+1}\nabla (r^nY_n^m) \quad \mbox{in} \quad D.
$$
Straightforward computation gives that
$$
\nabla\Scal_D[\bnu\cdot\mathbf{N}_{n+1}^m]=\frac{a_{n+1}^m }{2n+1}\nabla\Scal_D[(n+1)Y_n^m]=-\frac{n+1}{2n+1}\frac{a_{n+1}^m }{2n+1}\nabla (r^nY_n^m) \quad \mbox{in} \quad D.
$$
Then by using the jump formulas, there holds
$$
\nabla\times\Scal_D[\bnu\times\mathbf{N}_{n+1}^m]-\nabla\Scal_D[\bnu\cdot\mathbf{N}_{n+1}^m]=\frac{a_{n+1}^m }{2n+1}(\frac{1}{2}\nabla_S Y_n^m- \frac{1}{2}(n+1)Y_n^m\bnu)=-\frac{1}{2}\mathbf{N}_{n+1} \quad \mbox{on} \quad \p D.
$$
We next assume that
$$\mathbf{S}_D[\mathbf{N}_{n+1}^m]=c\mathcal{N}^m_{n+1}(\Bx)=c\left(a^m_{n+1} r^{n} Y^m_{n}(\hat{\Bx})\Bx + (1-\frac{a_{n+1}^m}{2n+1}-r^2) \nabla(r^{n} Y^m_{n}(\hat{\Bx}))\right),$$
 in $D$ and then one can show
$$
\frac{\partial}{\partial \bnu} \mathbf{S}_D[\mathbf{N}_{n+1}^m]|_{-}=c\left((\lambda+\mu)(n+2-\frac{2n}{a_{n+1}^m})+\lambda\right)\mathbf{N}_{n+1}^m
=c\mu(\frac{2(2n+1)}{a_{n+1}^m}-3)\mathbf{N}_{n+1}^m.
$$
Hence, there holds
$$
\mathbf{K}_D^*[\mathbf{N}_{n+1}^m]=-\frac{1}{2}\mathbf{N}_{n+1}^m-3\mu\mathbf{S}_D[\mathbf{N}_{n+1}^m]-\frac{1}{2n+3}(\frac{3}{2}-
\frac{(n+1)\mu}{2\mu+\lambda})\mathbf{N}_{n+1}^m.
$$
By using the jump formula, one has
$$
c\mu=-\frac{n\lambda+(3n+1)\mu}{(2n+3)(2n+1)(2\mu+\lambda)},
$$
and thus
$$
\mathbf{K}_D^*[\mathbf{N}_{n+1}^m]=\frac{-3\lambda+2\mu(2n^2+6n+1)}{2(2\mu+\lambda)(2n+1)(2n+3)}\mathbf{N}_{n+1}^m.
$$

Note that when $n=0$, both $\mathbf{T}_n^m$ and $\mathbf{M}_n^m$ vanish. By arranging the values of $n$, we complete the proof.

\end{proof}

\begin{rem}
Note that only the first eigenvalues $\xi_1^n$ in Theorem \ref{thm:eigenvalue} associated with the eigenfunctions $\mathbf{T}_n^m$ converge to zero as $n$ goes to infinity. These are the only possible spectra of the N-P operator $\mathbf{K}_D^*$ which can induce cloaking due to anomalous localised resonance.
\end{rem}

\section{Plasmonic resonance and cloaking}
In this section, we present our main results on cloaking due to anomalous localised resonance for the elastostatic system in the three dimensional case. In what follows, we let $B_r$ denote the central ball of radius $r\in\mathbb{R}_+$. Let $0<r_i<r_e<+\infty$. Set $(\tilde{\lambda}, \tilde{\mu})$ to be the Lam\'e constants in $B_{r_i}$ of the following form
\begin{equation}
  (\tilde{\lambda}, \tilde{\mu})= c_n (\lambda, \mu),
\end{equation}
where the coefficient $c_n>0$ depending on $n$ will be specified later, and $(\lambda,\mu)$ satisfies the convexity condition \eqref{eq:convex}. Set $(\breve{\lambda}, \breve{\mu})$ to be the Lam\'e constants in $B_{r_e}\backslash B_{r_i}$ of the following form
\begin{equation}
  (\breve{\lambda}, \breve{\mu})=(\epsilon_n + i \delta)(\lambda, \mu),
\end{equation}
where $\epsilon_n<0$ depends on $n$ and $\delta>0$. Define two elastic tensors $\widetilde{\mathbf{C}}(\Bx)=(\tilde{\mathrm{C}}_{ijkl}(\Bx))_{i,j,k,l=1}^3$ and $\breve{\mathbf{C}}(\Bx)=(\breve{\mathrm{C}}_{ijkl}(\Bx))_{i,j,k,l=1}^3$ as follows:
\begin{equation}
  \widetilde{\mathrm{C}}_{ijkl}(\Bx):=\widetilde{\lambda}(\Bx)\delta_{ij}\delta_{kl}+\tilde{\mu}(\Bx)(\delta_{ik}\delta_{jl}+\delta_{il}\delta_{jk}),\ \ \Bx\in  B_{r_i},
\end{equation}
and
\begin{equation}\label{eq:def_breve_C}
  \breve{\mathrm{C}}_{ijkl}(\Bx):=\breve{\lambda}(\Bx)\delta_{ij}\delta_{kl}+\breve{\mu}(\Bx)(\delta_{ik}\delta_{jl}+\delta_{il}\delta_{jk}),\ \ \Bx\in B_{r_e}\backslash B_{r_i}.
\end{equation}
With $\widetilde{\mathbf{C}}$ and $\breve{\mathbf{C}}$ defined above, and $\mathbf{C}$ defined in \eqref{eq:lame_constant}, we next introduce the elastic tensor $\mathbf{C}_B$ to be
\begin{equation}\label{eq:elasticity_tensor}
  \mathbf{C}_B := \widetilde{\mathbf{C}} {\chi_{B_{r_i}}} + \breve{\mathbf{C}} {\chi_{B_{r_e}\backslash B_{r_i}}} + \mathbf{C} {\chi_{\mathbb{R}^3\backslash B_{r_e}}},
\end{equation}
where $\chi$ denotes the characteristic function. Associated with the elastic material tensor in \eqref{eq:elasticity_tensor}, we consider the following transmission problem in elastostatics
\begin{equation}\label{eq:transmission_equation_prime}
\left\{
  \begin{array}{ll}
    \nabla \cdot \mathbf{C}_B \hat{\nabla}\bu_{\delta}=\mathbf{f}  & \mbox{in} \quad \mathbb{R}^3,\medskip \\
    \bu_{\delta}(\Bx)=\mathcal{O}(|\Bx|^{-1}) & \mbox{as} \quad |x|\rightarrow\infty,
  \end{array}
\right.
\end{equation}
where $\bff$ is a source function compactly supported in $\mathbb{R}^3\backslash  B_{r_e}$ and satisfies the condition \eqref{eq:source1}. To simplify the notation, we denote by $\Upsilon_i$ the boundary of $B_{r_i}$, i.e., $\partial B_{r_i}$ and $\Upsilon_e$ for $\partial B_{r_e}$. The operator $\partial_{\bnu_i}$ means taking the traction on the boundary $\Upsilon_i$ (see \eqref{eq:normald} for the definition) and it is same for $ \partial_{\bnu_e}$. The transmission problem is equivalent to the following system:
\begin{equation}\label{eq:transmission_equation_devided}
  \left\{
    \begin{array}{ll}
      \mathcal{L}_{\lambda, \mu}\bu_{\delta}(\Bx) =0 , & \mbox{in} \ \ B_{r_i} ,   \\
      \mathcal{L}_{\lambda, \mu}\bu_{\delta}(\Bx) =0 , & \mbox{in} \ \ B_{r_e}\backslash B_{r_i} ,   \\
      \mathcal{L}_{\lambda, \mu}\bu_{\delta}(\Bx) =\bff, &  \mbox{in} \ \ \mathbb{R}^3\backslash B_{r_e}, \\
      \bu_{\delta}|_- = \bu_{\delta}|_+, & \mbox{on} \ \ \Upsilon_i ,\\
      c_n\partial_{\bnu_i}\bu_{\delta}|_- = (\epsilon_n+i\delta)\partial_{\bnu_i}\bu_{\delta}|_+ , &  \mbox{on} \ \ \Upsilon_i, \\
      \bu_{\delta}|_- = \bu_{\delta}|_+, & \mbox{on} \; \Upsilon_e, \\
      (\epsilon_n+i\delta) \partial_{\bnu_e}\bu_{\delta}|_- = \partial_{\bnu_e}\bu_{\delta}|_+ , &  \mbox{on} \ \ \Upsilon_e .
    \end{array}
  \right.
\end{equation}
For analysis of anomalous localized resonance, we need consider the energy $\mathbf{E}_\delta$ defined in \eqnref{eq:energy2}, which is related to the solution in \eqnref{eq:transmission_equation_devided}.
To that end, we define
\begin{equation}\label{eq:source_potential}
  \mathbf{F}(\Bx) :=  \int_{\mathbb{R}^3} \mathbf{G}(\Bx-\By)\bff(\By) d\By,
\end{equation}
then the solution to the system \eqref{eq:transmission_equation_devided} can be represented as follows:
\begin{equation}\label{eq:solution_general_form}
  \bu_{\delta}(\Bx)=\mathbf{S}_{B_{r_i}}[\bvarphi_i] + \mathbf{S}_{B_{r_e}}[\bvarphi_e] + \mathbf{F},
\end{equation}
where $\bvarphi_i \in L^2(\Upsilon_i)^3$ and $\bvarphi_e \in L^2(\Upsilon_e)^3$. By using the transmission conditions, we can obtain the following equations
\begin{equation}
  \left\{
    \begin{split}
     (\epsilon_n + i \delta)\partial_{\bnu_i} \mathbf{S}_{B_{r_i}}[\bvarphi_i]\Big|_- - c_n \partial_{\bnu_i} \mathbf{S}_{B_{r_i}}[\bvarphi_i]\Big|_+ +& (\epsilon_n -c_n +i\delta)\partial_{\bnu_i}  \mathbf{S}_{B_{r_e}}[\bvarphi_e] \medskip\\
 =& (c_n-\epsilon_n -i\delta)\partial_{\bnu_i}  \mathbf{F},\medskip \\
     -(\epsilon_n + i \delta)\partial_{\bnu_e}  \mathbf{S}_{B_{r_e}}[\bvarphi_e]\Big|_- + \partial_{\bnu_e}  \mathbf{S}_{B_{r_e}}[\bvarphi_e]\Big|_+ +& (1-\epsilon_n  -i\delta)\partial_{\bnu_e}  \mathbf{S}_{B_{r_i}}[\bvarphi_i]\medskip \\
 =& (\epsilon_n-1 +i\delta)\partial_{\bnu_e}  \mathbf{F},
    \end{split}
  \right.
\end{equation}
that hold on $\Upsilon_i$ and $\Upsilon_e$, respectively.
By using the jump formula \eqnref{eq:jump_single} on $\Upsilon_i$ and $\Upsilon_e$ respectively, the above equations can be rewritten as:
\begin{equation}\label{eq:gernal_equation}
  \left[
    \begin{array}{cc}
      -a_{1,\delta} + \mathbf{K}^*_{\Upsilon_i} & \partial_{\bnu_i}  \mathbf{S}_{B_{r_e}} \\
      \partial_{\bnu_e}  \mathbf{S}_{B_{r_i}} & -a_{2,\delta} + \mathbf{K}^*_{\Upsilon_e} \\
    \end{array}
  \right]
\left[
  \begin{array}{c}
    \bvarphi_i \\
    \bvarphi_e \\
  \end{array}
\right] =
- \left[
    \begin{array}{c}
      \partial_{\bnu_i}  \mathbf{F} \\
      \partial_{\bnu_e}  \mathbf{F} \\
    \end{array}
  \right],
\end{equation}
where
\begin{equation}
  a_{1,\delta}=\frac{c_n+\epsilon_n+i\delta}{2(c_n-\epsilon_n-i\delta)} \quad \mbox{and} \quad  a_{2,\delta}=\frac{1+\epsilon_n+i\delta}{2(-1+\epsilon_n+i\delta)}.
\end{equation}
By using interior and exterior vector spherical harmonic functions and direct calculations, one has that
\begin{equation}\label{eq:single_h_1}
  \mathbf{S}_{B_{r_0}}[\mathbf{T}_n^m(\Bx)]=
\left\{
  \begin{array}{ll}
    \frac{d_1}{r^{n-1}_0}\mathcal{T}_n^m(\Bx) ,     & |\Bx|\leq r_0, \\
    d_1 r^{n+2}_0\nabla(r^{-(n+1)}Y_n^m)\times\Bx , & |\Bx|> r_0,
  \end{array}
\right.
\end{equation}
where
\begin{equation}
  d_1=\frac{-1}{\mu(2n+1)}.
\end{equation}
By \eqnref{eq:eigenpf05} it is easily found that the traction of $\mathcal{T}_1^m(\Bx)$ along the surface of any sphere vanishes, namely
\begin{equation}
  \frac{\partial}{\partial \bnu} (\mathcal{T}_1^m(\Bx)) =0,
\end{equation}
and hence in the following, we only consider the situation from $n\geq 2$.
If $\frac{\partial \mathbf{F}}{\partial \bnu_i}$ and $\frac{\partial \mathbf{F}}{\partial \bnu_e}$ are given as follows
\begin{equation}\label{eq:source_traction_coeff}
  \left[
    \begin{array}{c}
      \frac{\partial \mathbf{F}}{\partial \bnu_i} \\
      \frac{\partial \mathbf{F}}{\partial \bnu_e} \\
    \end{array}
  \right] =
\sum^{+\infty}_{n=2}\sum_{m=-n}^{n}
\left[
  \begin{array}{c}
    g^{n,m}_i \\
    g^{n,m}_e \\
  \end{array}
\right] \mathbf{T}_n^m(\Bx),
\end{equation}
by substituting \eqref{eq:single_h_1} into \eqref{eq:gernal_equation} and with the help of theorem \ref{thm:eigenvalue}, one has that the solution to the system \eqref{eq:gernal_equation} can be represented as follows:
\begin{equation}
\begin{split}
   & \bvarphi_i = \sum^{+\infty}_{n=2}\sum_{m=-n}^{n} \varphi^{n,m}_i \mathbf{T}_n^m(\Bx), \\
   & \bvarphi_e = \sum^{+\infty}_{n=2}\sum_{m=-n}^{n} \varphi^{n,m}_e \mathbf{T}_n^m(\Bx),
\end{split}
\end{equation}
where
\begin{equation}\label{eq:coeff_phi_phi_1}
  \begin{split}
      &  \varphi^{n,m}_i =- \frac{g^{n,m}_i(\xi^n_1 - a_{2,\delta}) - d_1 g^{n,m}_e \mu (n-1) (r_i/r_e)^{n-1}  } { (\xi^n_1-a_{1,\delta})(\xi^n_1 - a_{2,\delta}) + d^2_1 \mu^2(n-1)(n+2) (r_i/r_e)^{2n+1}  },  \\
      &   \varphi^{n,m}_e =- \frac{g^{n,m}_e(\xi^n_1 - a_{1,\delta}) + d_1 g^{n,m}_i \mu (n+2) (r_i/r_e)^{n+2}  } { (\xi^n_1-a_{1,\delta})(\xi^n_1 - a_{2,\delta}) + d^2_1 \mu^2(n-1)(n+2) (r_i/r_e)^{2n+1}  },
  \end{split}
\end{equation}
with $\xi^n_1$ given in \eqref{eq:eigenvalue}.

Since the source $\bff$ is located outside $B_{r_e}$, one has that $\mathbf{F}$ defined in \eqref{eq:source_potential} satisfies
\begin{equation}
   \mathcal{L}_{\lambda, \mu}\mathbf{F}(\Bx) =0, \quad    \Bx \in  B_{r_e},
\end{equation}
and $\mathbf{F}(\Bx)$ can be represented as follows:
\begin{equation}\label{eq:potential_F}
  \mathbf{F}(\Bx)= \sum^{\infty}_{n=2}\sum_{m=-n}^{n} \frac{g^{n,m}_e}{\mu(n-1)r_e^{n-1}}\mathcal{T}_n^m(\Bx) + c,
\end{equation}
for $|\Bx|\leq r_e$, where $g^{n,m}_e$ is given in \eqref{eq:source_traction_coeff}. Thus one can conclude that
\begin{equation}\label{eq:relation_gi_ge}
  g^{n,m}_i=(r_i/r_e)^{n-1} g^{n,m}_e.
\end{equation}
Substituting \eqref{eq:relation_gi_ge} into \eqref{eq:coeff_phi_phi_1}, we then get
\begin{equation}\label{eq:coeff_phi_phi_2}
  \begin{split}
      &  \varphi^{n,m}_i = \frac{ g^{n,m}_e(  a_{2,\delta} -\xi^n_1 + d_1  \mu (n-1) ) (r_i/r_e)^{n-1}  } { (\xi^n_1-a_{1,\delta})(\xi^n_1 - a_{2,\delta}) + d^2_1 \mu^2(n-1)(n+2) (r_i/r_e)^{2n+1}  },  \\
      &   \varphi^{n,m}_e =- \frac{g^{n,m}_e (\xi^n_1 - a_{1,\delta} + d_1  \mu (n+2) (r_i/r_e)^{2n+1} ) } { (\xi^n_1-a_{1,\delta})(\xi^n_1 - a_{2,\delta}) + d^2_1 \mu^2(n-1)(n+2) (r_i/r_e)^{2n+1}  }.
  \end{split}
\end{equation}
If we define
\begin{equation}\label{eq:plasmon_coeff}
  \begin{split}
     & c_{n_0}=(n_0+2)^2/(n_0-1)^2, \\
     & \epsilon_{n_0}=-1-3/(n_0-1),
  \end{split}
\end{equation}
with $n_0$ chosen properly later, then the denominator of $\varphi^{n_0,m}_i$ and $\varphi^{n_0,m}_e$ has the following relationship
\begin{equation}\label{eq:control_denominator}
  |(\xi^{n_0}_1-a_{1,\delta})(\xi^{n_0}_1 - a_{2,\delta}) + d^2_1 \mu^2(n_0-1)(n_0+2) (r_i/r_e)^{2n_0+1}| \approx \delta^2 + (r_i/r_e)^{2n_0}.
\end{equation}
With the help of \eqref{eq:single_h_1}, one has that
\begin{equation}\label{eq:single_outside}
  \mathbf{S}_{B_{r_i}}[\bvarphi_i] + \mathbf{S}_{B_{r_e}}[\bvarphi_e]= \sum_{n=2}^{\infty}\sum_{m=-n}^{n} d_1 \frac{r^{n+2}_e}{r^{n+1}}\left( \varphi^{n,m}_i +\varphi^{n,m}_e  \right)\mathbf{T}_n^m(\Bx) \quad |x|>r_e,
\end{equation}
and
\begin{equation}
  \begin{split}
     &  \mathbf{S}_{B_{r_i}}[\bvarphi_i] = \sum_{n=2}^{\infty}\sum_{m=-n}^{n} d_1 \frac{r^{n+2}_i}{r^{n+1}}\left( \varphi^{n,m}_i  \right)\mathbf{T}_n^m(\Bx), \quad r_i < |x| \leq r_e, \\
     &  \mathbf{S}_{B_{r_e}}[\bvarphi_e] = \sum_{n=2}^{\infty}\sum_{m=-n}^{n} d_1 \frac{r^n}{r^{n-1}_e}\left( \varphi^{n,m}_e  \right)\mathbf{T}_n^m(\Bx), \quad r_i < |x| \leq r_e.
  \end{split}
\end{equation}
Thus when $ r_i < |x| \leq r_e$, we denote
\begin{equation}\label{eq:single_devide}
    \mathbf{S}_{B_{r_i}}[\bvarphi_i] + \mathbf{S}_{B_{r_e}}[\bvarphi_e] = G_{\widetilde{n_0}} + G_{n_0},
\end{equation}
where
\begin{equation}
  \begin{split}
      & G_{\widetilde{n_0}} =  \sum_{n=2, n\neq n_0}^{\infty}\sum_{m=-n}^{n} d_1 \left( \frac{r^{n+2}_i}{r^{n+1}}\left( \varphi^{n,m}_i  \right) + \frac{r^n}{r^{n-1}_e}\left( \varphi^{n,m}_e  \right) \right) \mathbf{T}_n^m(\Bx), \\
      & G_{n_0}= \sum_{m=-n_0}^{n_0} d_1 \left( \frac{r^{n_0+2}_i}{r^{n_0+1}}\left( \varphi^{n_0,m}_i  \right) + \frac{r^{n_0}}{r^{n_0-1}_e}\left( \varphi^{n_0,m}_e  \right) \right) \mathbf{T}_{n_0}^m(\Bx).
  \end{split}
\end{equation}
Since the energy $ \int_{B_{r_e}\backslash B_{r_i}}|\nabla \mathbf{F}|^2 d\Bx < \infty $, the phenomenon
of the first equation in \eqnref{eq:condition_CALR} occurs  if and only if
\begin{equation}
   \bE(\bu_{\delta}-\mathbf{F})=\bE(\mathbf{S}_{B_{r_i}}[\bvarphi_i] + \mathbf{S}_{B_{r_e}}[\bvarphi_e])\rightarrow\infty,
\quad \mbox{as} \quad \delta\rightarrow 0.
\end{equation}
From \eqref{eq:single_devide},  one has that
\begin{equation}\label{eq:energy_deduce_1}
  \begin{split}
    \quad & \bE(\mathbf{S}_{B_{r_i}}[\bvarphi_i] + \mathbf{S}_{B_{r_e}}[\bvarphi_e])  \\
        = & \delta \left( \int_{B_{r_e}\backslash B_{r_i}} \hat{\nabla} (G_{\widetilde{n_0}} ) : \mathbf{C} \overline{\hat{\nabla}(G_{\widetilde{n_0}})} d\Bx  + \int_{B_{r_e}\backslash B_{r_i}} \hat{\nabla} (G_{n_0} ) : \mathbf{C}\overline{\hat{\nabla}(G_{n_0})} d\Bx \right)
  \end{split}
\end{equation}
By direct calculation, though a bit tedious, one can conclude that
\begin{equation}\label{eq:energy_deduce_2}
\begin{split}
  &\int_{B_{r_e}\backslash B_{r_i}} \hat{\nabla} (G_{\widetilde{n_0}} ) : \mathbf{C}\overline{\hat{\nabla}(G_{\widetilde{n_0}})} d\Bx  \\
  \leq  &\sum_{n=2, n\neq n_0}^{\infty}\sum_{m=-n}^{n} \frac{|g^{n,m}_e|^2}{n} \left( \frac{n^2}{(n-n_0)^2} + \frac{n^4}{(n-n_0)^4} \left(\frac{r_i}{r_e} \right)^{2n}  \right) \leq C,
\end{split}
\end{equation}
where $C$ is independent of $\delta$. With the help of \eqref{eq:control_denominator}, one can obtain 
\begin{equation}\label{eq:energy_deduce_3}
 \int_{B_{r_e}\backslash B_{r_i}} \hat{\nabla} (G_{n_0}) : \mathbf{C}\overline{ \hat{\nabla}(G_{n_0})} d\Bx \approx \sum_{m=-n_0}^{n_0} \frac{|g^{n_0,m}_e|^2}{n_0(\delta^2+(r_i/r_e)^{2n_0})}.
\end{equation}
Finally, combining \eqref{eq:energy_deduce_1}, \eqref{eq:energy_deduce_2} and \eqref{eq:energy_deduce_3}, we have
\begin{equation}\label{eq:energy_estimate}
   \bE(\bu_{\delta}) \approx \sum_{m=-n_0}^{n_0}  \frac{\delta |g^{n_0,m}_e|^2}{n_0(\delta^2+(r_i/r_e)^{2n_0})}.
\end{equation}

We are now in the position of presenting the main theorem on the cloaking due to anomalous localised resonance result on the three dimensional elastostatic system. We define the critical radius by
\begin{equation}
  r_*=\sqrt{r_e^3/r_i}.
\end{equation}
\begin{thm}
Let the elasticity tensor $\mathbf{C}_B$ be given \eqref{eq:elasticity_tensor} with $c_{n_0}$ and $\epsilon_{n_0}$ given in \eqref{eq:plasmon_coeff}. If the source $\mathbf{f}$ is supported in $r_e<|\Bx|<r_*$. Then cloaking due to anomalous localised resonance occurs, namely, the condition \eqref{eq:condition_CALR} is satisfied. Moreover, if the source $\bff$ is supported outside $B_{r_*}$, then resonance does not occur, namely $\bE(\bu_{\delta})<\infty$.
\end{thm}
\begin{proof}
We first prove the second condition in \eqref{eq:condition_CALR}. In other words, $\bu_{\delta}(\Bx)$ is bounded outside some region. From the expression \eqref{eq:coeff_phi_phi_2} and approximation \eqref{eq:control_denominator}, one has that
\begin{equation}
\begin{split}
  |\varphi^{n_0,m}_i + \varphi^{n_0,m}_e|  &  \leq  C g^{n_0,m}_e \frac{(r_i/r_e)^n + \delta}{(r_i/r_e)^{2n}+\delta^2} \\
    & =C g^{n_0,m}_e \left(  \frac{(r_i/r_e)^n }{(r_i/r_e)^{2n}+\delta^2} + \frac{1}{\frac{(r_i/r_e)^{2n}}{\delta}+\delta}  \right)\\
    & \leq C g^{n_0,m}_e \left(  \frac{(r_i/r_e)^n }{(r_i/r_e)^{2n}} + \frac{1}{\frac{(r_i/r_e)^{2n}}{\delta}+\delta}  \right) \\
    & \leq C g^{n_0,m}_e  \frac{1}{(r_i/r_e)^{n}},
\end{split}
\end{equation}
where the constant $C$ may change from one inequality to another. If $n\neq n_0$, one has that
\begin{equation}
  \begin{split}
  |\varphi^{n,m}_i + \varphi^{n,m}_e|  &  \leq  C g^{n,m}_e \left( \frac{n}{n-n_0} +\frac{n^2}{(n-n_0)^2}(r_i/r_e)^n \right) \\
   & \leq C g^{n,m}_e.
\end{split}
\end{equation}
From the expression \eqref{eq:single_outside} outside $B_{r_e}$, one has that if $|x|>r_e^2/r_i$,
\begin{equation}
  |\bu_{\delta}(\Bx)| \leq |\mathbf{F}| + C  \sum_{n=2}^{\infty}\sum_{m=-n}^{n} d_1 g^{n,m}_e  \frac{r^{n+2}_e}{r^{n+1}} \frac{1}{(r_i/r_e)^{n}} \leq C,
\end{equation}
where the constant $C$ depends only on the source $\bff$. Thus the second condition in \eqref{eq:condition_CALR} is satisfied. Next we consider the energy $ \bE(\bu_{\delta}) $. Let $n_0$ be chosen such that
\begin{equation}\label{choose_n_0}
   (r_i/r_e)^{n_0}< \delta \leq (r_i/r_e)^{n_0-1},
\end{equation}
and hence from the expression \eqref{eq:energy_estimate}, one has that
\begin{equation}
  \begin{split}
    \bE(\bu_{\delta}) &  \approx \sum_{m=-n_0}^{n_0}  \frac{\delta |g^{n_0,m}_e|^2}{n_0(\delta^2+(r_i/r_e)^{2n_0})} \\
                  & \geq  \frac{C }{n_0 (r_i/r_e)^{n_0} } \sum_{m=-n_0}^{n_0} |g^{n_0,m}_e|^2\\
                  & \geq C \frac{r_e^{n_0}}{r_i^{n_0}} \frac{1}{n_0(2n_0+1)} \left( \sum_{m=-n_0}^{n_0} |g^{n_0,m}_e| \right)^2
  \end{split}
\end{equation}
Since $\bff$ is supported inside $r_*$, the potential $\mathbf{F}$ given in \eqref{eq:potential_F} can not converge at $|x|=r_*$. Then the following holds
\begin{equation}
\limsup_{n\rightarrow \infty}\left(\sum_{m=-n}^{m=n} \frac{|g^{n,m}_e| }{n r_e^{n-1}} \right)^{1/n} > 1/\sqrt{\frac{r_e^3}{r_i}},
\end{equation}
namely,
\begin{equation}
\limsup_{n\rightarrow \infty}\left(\sum_{m=-n}^{m=n} |g^{n,m}_e|  \right)^2 > C n^2 \frac{r_i^n}{r_e^n}.
\end{equation}
Finally, we have that
\begin{equation}
  \sup \bE(\bu_{\delta})\rightarrow \infty \quad \mbox{as} \quad \delta\rightarrow 0.
\end{equation}
If the source $\bff$ is supported outside the ball $B_{r_*}$, then the potential $\mathbf{F}$ given in \eqref{eq:potential_F} converges at $|\Bx|=r_*+\tau$, for sufficiently small $\tau\in\mathbb{R}_+$. With $n_0$ again chosen in \eqref{choose_n_0} and from the expression \eqref{eq:energy_estimate}, one has that
\begin{equation}
  \bE(\bu_\delta) \leq C \frac{C }{n_0 (r_i/r_e)^{n_0} } \sum_{m=-n_0}^{n_0} |g^{n_0,m}_e|^2 \leq C \norm{ \bff}_{L^2(\mathbb{R}^3)^3}^2.
\end{equation}

The proof is complete.
\end{proof}

\section*{Acknowledgment}
The work of Y. Deng was supported by NSF grant of China, No. 11601528, Mathematics and Interdisciplinary Sciences Project, Central South University, and Innovation Program of Central South University, No. 10900-506010101. The work of H. Liu was supported by the startup fund and FRG grants from Hong Kong Baptist University, Hong Kong RGC General Research Funds, 12302415, and the NSF grant of China, No. 11371115.

\end{document}